\documentclass[envcountsame]{svmult}

\usepackage[utf8]{inputenc}

\usepackage{amssymb}
\usepackage{url,textcmds}
\usepackage{mathtools,amscd,enumerate}
\usepackage[british]{babel}

\newcommand{\itref}[1]{\textup{(\ref{#1})}}

\newcommand{\hyphen}{-\nobreak\hskip 0pt plus 0pt}

\newcommand{\Lie}[1]{\mathrm{#1}}
\newcommand{\lie}[1]{\mathfrak{#1}}
\newcommand{\GL}{\Lie{GL}}
\newcommand{\gl}{\lie{gl}}
\newcommand{\sln}{\lie{sl}}
\newcommand{\Sl}{\sln}
\newcommand{\SO}{\Lie{SO}}

\newcommand{\SP}{\Lie{Sp}}

\newcommand{\su}{\lie{su}}
\newcommand{\un}{\lie u}

\newcommand{\bmf}{\lie b}

\newcommand{\kf}{\lie k}
\newcommand{\n}{\lie n}

\newcommand{\tf}{\lie t}

\newcommand{\SU}{\Lie{SU}}
\newcommand{\Un}{\Lie{U}}

\newcommand{\bC}{\mathbb C}
\newcommand{\bH}{\mathbb H}
\newcommand{\bK}{\mathbb K}
\newcommand{\bL}{\mathbb L}
\newcommand{\bN}{\mathbb N}
\newcommand{\bR}{\mathbb R}
\newcommand{\bZ}{\mathbb Z}

\newcommand{\C}{\bC}
\newcommand{\R}{\bR}
\newcommand{\HH}{\bH}
\newcommand{\Z}{\bZ}

\newcommand{\any}{{\,\cdot\,}}
\newcommand{\Hodge}{{*}}
\newcommand{\hook}{{\lrcorner}}

\newcommand{\symp}{{/\mkern-3mu/}}
\newcommand{\hkq}{{/\mkern-3mu/\mkern-3mu/}}

\newcommand{\Mcut}{M_{\mathrm{cut}}}
\newcommand{\Mmod}{M_{\mathrm{mod}}}

\DeclarePairedDelimiter{\abs}{\lvert}{\rvert}
\DeclarePairedDelimiterX{\inp}[2]{\langle}{\rangle}{#1, #2}
\DeclarePairedDelimiter{\norm}{\lVert}{\rVert}

\newcommand{\with}{}
\newcommand{\SetSymbol}[1][]{\nonscript\:#1\vert
  \allowbreak\nonscript\:\mathopen{}}
\DeclarePairedDelimiterX{\Set}[1]{\{}{\}}{%
  \renewcommand{\with}{\SetSymbol[\delimsize]}
#1 }

\DeclareMathOperator{\diag}{diag}
\DeclareMathOperator{\Hom}{Hom}
\DeclareMathOperator{\imp}{Im} 
\DeclareMathOperator{\imm}{im} 
\DeclareMathOperator{\Mor}{Mor}
\DeclareMathOperator{\nLie}{Lie}
\DeclareMathOperator{\Span}{Span}
\DeclareMathOperator{\stab}{stab}

\newcommand{\iH}{\imp\bH}
\newcommand{\pr}{\mathrm{pr}}

\begin{document}

\title*{Hypertoric manifolds and hyperKähler moment maps}

\author{Andrew Dancer and Andrew Swann}

\institute{Andrew Dancer \at Jesus College, Oxford, OX1 3DW, United
Kingdom \email{dancer@maths.ox.ac.uk} \and Andrew Swann \at Department
of Mathematics, Aarhus University, Ny Munkegade 118, Bldg 1530,
DK-8000 Aarhus C, Denmark \email{swann@math.au.dk}}

\maketitle

\begin{center}
  \emph{To Simon Salamon on the occasion of his 60th birthday}
\end{center}

\abstract{ We discuss various aspects of moment map geometry in
symplectic and hyperK\"ahler geometry. In particular,
we classify complete hyperKähler manifolds of dimension
\( 4n \) with a tri-Hamiltonian action of a torus of
dimension~\( n \), without any assumption on the finiteness of the
Betti numbers.  As a result we find that the hyperKähler moment in
these cases has connected fibres, a property that is true for
symplectic moment maps, and is surjective. New examples
of hypertoric manifolds of infinite topological type are produced.
 We provide examples of
non-Abelian tri-Hamiltonian group actions of connected groups on
complete hyperKähler manifolds such that the hyperKähler moment map is
not surjective and has some fibres that are not connected.  We also
discuss relationships to symplectic cuts, hyperKähler modifications
and implosion constructions.}

\section{Introduction}
\label{sec:introduction}

A symplectic structure on a (necessarily even-dimensional) manifold is
a closed non-degenerate two-form.  Several Riemannian and
pseudo-Riemannian geometries have been developed over the years which
give rise to a symplectic structure as part of their data.  The most
famous example is that of a \emph{hyperK\"ahler structure}, where we
have a Riemannian metric $g$ and complex structures $I$, $J$, $K$
obeying the quaternionic multiplication relations, and such that $g$
is K\"ahler with respect to $I,J,K$. We therefore obtain a triple
$(\omega_I, \omega_J, \omega_K)$ of symplectic forms.

One of the foundational results of symplectic geometry is the Darboux
Theorem, which says that locally a symplectic structure can be put
into a standard form $\omega = \sum_{i=1}^{n} dp_i \wedge dq_i$.  Many
of the interesting questions in symplectic geometry are therefore
global in nature, giving the subject a more topological flavour.

Geometries involving a metric of course do not have a Darboux-type
theorem, because the metric contains local information through its
curvature tensor.  However, there is one area of symplectic geometry,
that concerning \emph{moment maps} where a rich theory has been
developed for other geometries by analogy with the symplectic
situation.  In this paper we shall discuss some aspects of this,
especially related to hypertoric manifolds, cutting and implosion.

\section{Hypertoric manifolds}
\label{sec:toric-manifolds}

Let \( M \) be a hyperKähler manifold \( M \) of dimension~\( 4n \).
We say that an action of a group \( G \) on~\( M \) is
\emph{tri-symplectic} if it preserves each of the symplectic forms
\( \omega_I \), \( \omega_J \) and \( \omega_K \).  This is equivalent
to \( G \) preserving both the metric \( g \) and each of the
associated complex structures \( I \), \( J \) and \( K \); so the
action is isometric and \emph{tri-holomorphic}.
We will usually assume that \( G \) is connected and that the action
is effective.

Because hyperKähler metrics are Ricci-flat, we have that if \( M \) is
compact, then any Killing field~\( X \) is parallel and so \( G \) is
Abelian.  As the complex structures are also parallel the distribution
\( \bH X = \Span_{\bR}\Set{X,IX,JX,KX} \) is integrable and flat.  Up
to finite covers, an \( M \) is a product \( T^{4m} \times M_0 \),
with \( G \) acting trivially on \( M_0 \).

Thus the interesting cases are when \( M \) is non-compact.  From the
Riemannian perspective is now natural to consider complete metrics.
Note that by Alekseevski\u\i\ \& Kimel'fel'd
\cite{Alekseevskii-K:Ricci}, any homogeneous hyperKähler manifold is
flat; such a manifold is necessarily complete, so its universal cover
is~\( \bR^{4n} \) with the flat metric.  Thus one should consider
actions on~\( M \) with orbits of dimension strictly less than~\( 4n \).

One says that a tri-holomorphic action of \( G \) on~\( M \) is
\emph{tri-Hamiltonian} if it is Hamiltonian for each symplectic
structure, meaning that there are equivariant moment maps
\begin{gather}
  \mu_I,\mu_J,\mu_K \colon M \to \lie g^*,\\
  \label{eq:hook}
  d\mu_A^X = X \hook \omega_A,
\end{gather}
where \( \mu_A^X = \inp{\mu_A}X \).  Here we write \( X \) both for
the element of~\( \lie g \) and the corresponding vector field
\( x\mapsto X_x \) on~\( M \).  Also \( \inp \alpha X = \alpha(X) \)
is the pairing between \( \lie g^* \) and \( \lie g \).

As each \( X \in \lie g \) preserves \( \omega_A \), we have
\( 0 = L_X\omega_A = X \hook d\omega_A + d(X\hook\omega_A) =
d(X\hook\omega_A) \), so \( X\hook\omega_A \) is exact.  Thus if
\( M \) is simply-connected then, equation~\eqref{eq:hook} has a
solution \( \mu_A^X \in C^\infty(M) \) that is unique up to an
additive constant.

\subsection{Abelian actions}
\label{sec:abelian-actions}

For an Abelian group~\( G \), equivariance of \( \mu_A \) is the just
the condition \( L_X\mu_A^Y = 0 \) for each \( X,Y \in \lie g \).  But
\( L_X\mu_A^Y = X\hook d\mu_A^Y = - \omega_A(X,Y) = - g(AX,Y) \) and
\( d(\omega_A(X,Y)) = L_Y(X\hook\omega_A) = 0 \).  So \( L_X\mu_A^Y \)
is constant and the action is tri-Hamiltonian only if for each~\( A \)
we have \( \mathcal G \bot A\mathcal G \), where
\( \mathcal G_x = \Set{X_x \with X \in \lie g} \subset T_xM \).  This
last condition is equivalent to
\( \dim \bH \mathcal G_x = 4\dim \mathcal G_x \) for each
\( x \in M \).

\begin{proposition}
  Suppose \( G \) is a connected Abelian group that has an effective
  tri-Hamiltonian action on a connected hyperKähler manifold~\( M \)
  of dimension~\( 4n \).  Then the dimension of \( G \) is at
  most~\( n \).
\end{proposition}

\begin{proof}
  For each \( x \in M \), the discussion above shows that the
  tri-Hamiltonian condition gives \( \dim\mathcal G_x \leqslant n \).
  We thus need to show that there is some \( x \in M \) such that the
  map \( \lie g \to \mathcal G_x \), \( X \mapsto X_x \), is
  injective.

  Fix a point \( x \in M \) such that \( \dim\stab_G(x) \) is the
  least possible.  Note that \( H = \stab_G(x) \) is a compact
  subgroup of \( \SP(n) \leqslant \SO(4n) \).  We may therefore
  \( H \)-invariantly write \( T_xM = T_x(G\cdot x) \oplus W \) as an
  orthogonal direct sum of the tangent space to the orbit
  through~\( x \) and its orthogonal complement~\( W \).  Now consider
  the map \( F\colon G\times W \to M \) given by
  \begin{equation*}
    F(g,w) = g\cdot (\exp_xw) = \exp_{gx}{g_*w}.
  \end{equation*}
  At \( (e,0) \in G \times W \) this has differential
  \( (F_*)_{(e,0)}(X,w) = X_x + w \) and
  \( F(gh,(h_*)^{-1}w) = F(g,w) \) for each \( h \in H \).  Thus
  \( F \) descends to a diffeomorphism from a neighbourhood of
  \( (e,0) \in G\times_H W \) to a neighbourhood~\( U \) of
  \( x \in M \) which is equivariant for the action of~\( \lie g \).
  In particular \( \stab_G F(e,w) \subset H \) when \( F(e,w) \in U \).

  As \( G \) acts effectively, we have for each
  \( X \in \lie g\setminus \Set0 \) there is some point~\( y \) with
  \( X_y \ne 0 \).  But \( M \) is Ricci-flat, so the Killing vector
  field \( X \) is analytic, thus the set
  \( \Set{ y \in M \with X_y \ne 0} \) is open and dense.

  If \( \dim H = \dim\stab_G(x) \) is non-zero, then there is a
  non-zero element \( X \in \lie h \).  Now \( X \) is non-zero at
  some point \( y = F(g,w) \) of~\( U \), and
  \( z = g^{-1}y = F(e,w) \) has \( X_z = (g_*)^{-1}X_y \ne 0 \) too,
  since \( G \) is Abelian.  So \( \nLie\stab_G(z) \) is a subspace of
  \( \lie h \) not containing~\( X \).  It follows that
  \( \dim\stab_G(z) < \dim\stab_G(x) \), contradicting our choice
  of~\( x \).

  We conclude that \( \dim\lie h = 0 \), so \( \stab_G(x) \) is finite
  and \( \lie g \mapsto \mathcal G_x \) is a bijection.  Thus \(
  \dim\lie g \leqslant n \).
  \qed
\end{proof}

Any connected Abelian group of finite dimension is of the form
\( G = \bR^m \times T^k \) for some \( m,k\geqslant 0 \).  If \( M \)
is simply-connected then the tri-symplectic \( T^k \)-action is
necessarily tri-Hamiltonian: each \( \mu_A^Y \) obtains its maximum on
each \( T^k \)-orbit, and so \( L_X\mu_A^Y = X \hook d\mu_A^Y \) is
zero at these points, and hence on all of~\( M \).  If
\( \dim G = n \) and the \( G \)-action is tri-Hamiltonian, Bielawski
\cite{Bielawski:tri-Hamiltonian} proves that the \( \bR^m \) factor
acts freely and any discrete subgroup of~\( \bR^m \) acts properly
discontinuously, so a discrete quotient of~\( M \) has a
tri-Hamiltonian~\( T^n \)-action.  In general, a hyperKähler manifold
of dimension~\( 4n \) with a tri-Hamiltonian \( T^n \) action is
called \emph{hypertoric}.

Bielawski \cite{Bielawski:tri-Hamiltonian} classified the hypertoric
manifolds in any dimension under the assumption that \( M \) has
finite topological type, meaning that the Betti numbers of~\( M \) are
finite.  For \( \dim M=4 \), this classification is extended to
general hypertoric~\( M \) in \cite{Swann:twist-mod}.  Here we wish to
provide the full classification of hypertoric manifolds in arbitrary
dimension, without any restriction on the topology.  First let us
recall some of the four-dimensional story.

\subsection{Dimension four}
\label{sec:dimension-four}

Let \( M \) be a four-dimensional hyperKähler manifold with
an effective tri-Hamiltonian \( S^1 \)-action of period~\( 2\pi \).
Let \( X \) be the corresponding vector field on~\( M \).  Note that
the only special orbits for the action are fixed points: if
\( g \in S^1 \) stabilises the point~\( x \) and \( X_x \ne 0 \), then
\( g \) fixes \( T_xM = \Span\Set{X_x,IX_x,JX_x,KX_x} \) and hence a
neighbourhood of~\( x \), so by analyticity \( g = e \).

The hyperKähler moment map
\begin{equation*}
  \mu = (\mu_I,\mu_J,\mu_K) \colon M \to \bR^3
\end{equation*}
is a local diffeomorphism away from the fixed point set~\( M^X \).
Locally on \( M' = M\setminus M^X \), the hyperKähler metric may be
written as
\begin{equation*}
  g = \frac1V \beta_0^2 + V(\alpha_I^2+\alpha_J^2+\alpha_K^2),
\end{equation*}
where \( \alpha_A = X\hook\omega_A = d\mu_A \), \( V = 1/g(X,X) \) and
\( \beta_0 = \alpha_0/\norm X = g(X,\any)V^{1/2} \).  The hyperKähler
condition is now equivalent to the monopole equation
\( d\beta_0 = - \Hodge_3 dV \), which implies that locally \( V \) is
a harmonic function on~\( \bR^3 \).

\begin{theorem}[\cite{Swann:twist-mod}]
  \label{thm:4d}
  Let \( M \) be a complete connected hyperKähler manifold of
  dimension~\( 4 \) with a tri-Hamiltonian circle action of
  period~\( 2\pi \).  Then the hyperKähler moment map
  \( \mu\colon M \to \bR^3 \) is surjective with connected fibres and
  induces a homeomorphism \( \overline\mu\colon M/S^1 \to \bR^3 \).
  The metric on~\( M \) is specified by any harmonic
  function~\( V\colon \bR^3\setminus Z \to (0,\infty) \) of the form
  \begin{equation}
    \label{eq:V3}
    V(p) = c + \frac12 \sum_{q \in Z} \frac1{\norm{p - q}}
  \end{equation}
  where \( c \geqslant 0 \) is constant and \( Z \subset \bR^3 \)~is
  finite or countably infinite.
  \qed
\end{theorem}

The set \( Z = \mu(M^X)\) is the image of the fixed-point
set~\( M^X \).  The metrics with \( c = 0 \) and \( Z \)~finite are
the Gibbons-Hawking metrics \cite{Gibbons-H:multi}; \( c > 0 \) and
\( Z \)~finite gives the older multi-Taub-NUT metrics
\cite{Hawking:gravitational}.  For \( Z \ne \varnothing \), the
hyperKähler manifold \( M \) is simply-connected, with
\( b_2(M) = \abs Z - 1 \); for \( c > 0 \), \( Z = \varnothing \), we
have \( M = S^1 \times \bR^3 \) with \( \mu \)~projection to the
second factor.  For \( Z \) infinite and \( c = 0 \), the hyperKähler
metrics are of type \( A_\infty \) as constructed by Anderson,
Kronheimer and LeBrun~\cite{AndersonMT-KL:infinite} and
Goto~\cite{Goto:A-infinity}, concentrating on the case
\( Z = \Set{(n^2,0,0) \with n \in \bN_{>0}} \), and written down for
general \( Z \) in Hattori~\cite{Hattori:Ainfty-volume}.  These are
the examples of infinite topological type.

To understand and extend Hattori's general formulation of these
structures, we need to study when \eqref{eq:V3} gives a finite sum on
an open subset of~\( \bR^3 \).  By Harnack's Principle (see
e.g.~\cite{Axler-BR:harmonic}) if \eqref{eq:V3} is finite at one point
of~\( \bR^3 \), then it is finite on all of \( \bR^3\setminus Z \).
This may be seen in a elementary way via the following result.

\begin{lemma}
  Suppose \( (q_n)_{n\in\bN} \) is a sequence of points in~\( \bR^3 \)
  is given. Then the series
  \( S_1 = \sum_{n \in \bN} \norm{p-q_n}^{-1} \) converges at some
  \( p \in \bR^3\setminus \Set{q_n\with n\in\bN} \) if and only if the
  series \( S_2 = \sum_{n \in \bN}(1+\norm{q_n})^{-1} \) converges.
\end{lemma}

\begin{proof}
  First note that if there is a compact subset \( C \) of~\( \bR^3 \)
  containing infinitely many points of the sequence \( (q_n) \), then
  neither sum converges: there is some subsequence \( (q_i)_{i\in I} \) that
  converges to a~\( q \in \bR^3 \), and so infinitely many terms are
  greater than some strictly positive lower bound.

  Now putting \( c = 1 + \norm p \), we have
  \( \norm{p - q} \leqslant \norm p + \norm q \leqslant c(1 + \norm q)
  \).  It follows that convergence of \( S_1 \) implies converges
  of~\( S_2 \).

  For the converse, we consider
  \( q \in \bR^3 \setminus \overline B(0;R) \) for
  \( R = 1+2\norm p \) and have
  \begin{equation}
    \label{eq:ineq}
    \begin{split}
      \norm{p - q}
      &\geqslant \norm q - \norm p = \tfrac12 \norm q + (\tfrac
      12\norm q - \norm p)\\
      &> \tfrac12(\norm q + 1).
    \end{split}
  \end{equation}
  If \( S_2 \) converges, then
  \( \Set{n \in \bN \with q_n \in \overline B(0;R)} \) is
  finite, so the inequality~\eqref{eq:ineq} implies convergence
  of~\( S_1 \).
  \qed
\end{proof}

Finally let us remark that scaling the hyperKähler metric~\( g \) by a
constant~\( C \) scales~\( V \) as a function on~\( M \) by
\( C^{-1} \).  However the hyperKähler moment map~\( \mu \) also
scales by~\( C \), so the induced function~\( V(p) = V(\mu(x)) \)
on~\( \bR^3 \) has the same form, with a new constant term~\( c/C \)
and the points \( q \) replaced by~\( q/C \).  On the other hand
scaling the vector field \( X \) by a constant, so that the action it
generates is no longer of period~\( 2\pi \), scales \( V \) on~\( M \)
and \( \mu \) by different weights.  In particular, such a change
alters the factors \( 1/2 \) in~\eqref{eq:V3}.

\subsection{Construction of hypertoric manifolds}
\label{sec:constr-hypert-manif}

Bielawski and Dancer \cite{Bielawski-Dancer:toric} provided a general
construction of hypertoric manifolds in all dimensions with finite
topological type.  Goto \cite{Goto:A-infinity} gave a particular
construction of examples of in arbitrary dimension of infinite
topological type.  Let us now build on Hattori's four-dimensional
description \cite{Hattori:Ainfty-volume}, to combine these two
constructions.

Let \( \bL \) be a finite or countably infinite set.
Choose \( \Lambda = (\Lambda_k)_{k\in\bL} \in \bH^\bL \) and define
\( \lambda = (\lambda_k)_{k\in \bL} \) by
\( \lambda_k = - \tfrac12\overline\Lambda_k i \Lambda_k \in \iH
\).  For each \( k \in \bL \), let \( u_k \in \bR^n \) be a non-zero
vector and put \( \hat\lambda_k = \lambda_k/\norm{u_k} \).  Suppose
\begin{equation}
  \label{eq:convergence}
  \sum_{k\in \bL}(1 + \abs{\hat\lambda_k})^{-1} < \infty.
\end{equation}
Consider the Hilbert manifold
\( M_\Lambda = \Lambda + \bL^2(\bH) \), where
\begin{equation*}
  \bL^2(\bH) = \Set[\Big]{ v \in
  \bH^{\bL} \with \sum_{k\in\bL} \abs{v_k}^2 < \infty}.
\end{equation*}
Let \( T_\lambda \) be the Hilbert group
\begin{equation*}
  T_\lambda = \Set[\Big]{ g \in T^{\bL} =
  (S^1)^{\bL} \with \sum_{k\in\bL} (1 +
  \abs{\lambda_k})\, \abs{1-g_k}^2 < \infty}.
\end{equation*}
If \( \norm{u_k} \) is bounded away from~\( 0 \), then
\( g \in T_\lambda \) implies \( g_k \) is arbitrarily close to~\( 1 \)
except for a finite number of \( k \in \bL \).  As
\( \abs{1-\exp(it)}^2 = 2 - 2\cos(t) \leqslant t^2 \) for all
\( t\in\bR \) and \( \abs{1-\exp(it)}^2 \geqslant 2t^2/\pi^2 \) on
\( (-\pi/2,\pi/2) \), we see that the Lie algebra of \( T_\lambda \)
is
\begin{equation*}
  \lie t_\lambda = \Set[\Big]{ t \in \bR^\bL \with \norm
  t_{\lambda,\lie t}^2 = \sum_{k\in\bL} (1 + \abs{\lambda_k})\, \abs{t_k}^2
  < \infty}.
\end{equation*}

Now consider the linear map
\( \beta \colon \lie t_\lambda \to \bR^n \) given by
\( \beta(e_k) = u_k \), where \( e_i = (\delta_i^k)_{k\in\bL} \),
where \( \delta_i^k \in \Set{0,1} \) is Kronecker's delta.  Supposing
\( \beta \)~is continuous then we define
\( \lie n_\beta = \ker \beta \subset \lie t_\lambda\).  If
\( u_k \in \bZ^n \subset \bR^n \) for each~\( k\in\bL \), we may
define a Hilbert subgroup \( N_\beta \) of~\( T_\lambda \) by
\begin{equation*}
  N_\beta = \ker(\exp \circ \beta \circ \exp^{-1} \colon T_\lambda \to
  T^n).
\end{equation*}
This gives exact sequences
\begin{gather}
  \label{eq:ex-lie}
  \begin{CD}
    0 @>>> \lie n_\beta @>\iota>> \lie t_\lambda @>\beta>> \bR^n @>>> 0,
  \end{CD}\\
  \label{eq:ex-grp}
  \begin{CD}
    0 @>>> N_\beta @>>> T_\lambda @>>> T^n @>>> 0.
  \end{CD}
\end{gather}
Our aim now is to construct hypertoric manifolds of dimension~\( 4n \)
as hyperKähler quotients of \( M_\Lambda \) by~\( N_\beta \).

\begin{remark}
  The construction of Hattori \cite{Hattori:Ainfty-volume} corresponds
  to \( n = 1 \) and \( u_k = 1 \in \bR \) for each~\( k \).  For
  general dimension~\( 4n \), Goto's construction
  \cite{Goto:A-infinity} corresponds to
  \( \bL = (\bZ\setminus\Set0) \amalg \Set{1,\dots,n} \),
  \begin{equation*}
    \Lambda_k =
    \begin{dcases*}
      k\mathbf i,&for \( k \in \bZ_{>0} \),\\
      k\mathbf k,&for \( k \in \bZ_{<0} \),\\
      0,&for \( k \in \Set{1,\dots,n} \),
    \end{dcases*}
    \quad\text{with}\quad
    u_k =
    \begin{dcases*}
      \mathbf e_1,&for \( k \in \bZ\setminus\Set{0} \),\\
      \sum_{i=1}^n \mathbf e_i,&for \( k = 1 \in \Set{1,\dots,n} \),\\
      - \mathbf e_r,&for \( k = r \in \Set{2,\dots,n} \).
    \end{dcases*}
  \end{equation*}
  Thus Goto's construction is for one concrete choice of \(
  (\lambda_k)_{k\in\bL} \) and only one of the \( u_k \)'s is
  repeatedly infinitely many times.
\end{remark}

Returning to the general situation, note that the integrality of
\( u_k \) implies \( \norm{u_k} \geqslant 1 \), so the convergence
condition~\eqref{eq:convergence} implies
\begin{equation}
  \label{eq:conv-lambda}
  \sum_{k\in\bL} (1+\abs{\lambda_k})^{-1} < \infty.
\end{equation}

The group \( T_\lambda \) acts on \( M_\Lambda \) via
\( gx = (g_kx_k)_{k\in\bL} \): indeed for \( g \in T_\lambda \) and
\( x = \Lambda + v \in M_\Lambda \), we have
\( gx = g\Lambda + gv = \Lambda - (1-g)\Lambda + gv \),
but \( gv \in \bL^2(\bH) \) and
\( \norm{(1-g)\Lambda}^2 = \sum_{k\in\bL}
\tfrac12\abs{\lambda_k}\abs{1-g_k}^2 \leqslant \tfrac12
\sum_{k\in\bL} (1+\abs{\lambda_k}) \abs{1-g_k}^2 \), which is finite
by the definition of~\( T_\lambda \), so \( (1-g)\Lambda \in \bL^2(\bH)
\) too.  The action preserves the flat
hyperKähler structure with \( \bL^2 \)-metric and complex structures
obtained by regarding \( \bL^2(\bH) \) as a right \( \bH \)-module.
Identifying \( \bR^3 \) with
\( \iH = \bR\mathbf i+ \bR\mathbf j+\bR\mathbf k \), a corresponding
hyperKähler moment map is given by
\( \inp{\mu_\Lambda(x)}{t} = \tfrac12\sum_{k\in\bL}
(\overline x_k\mathbf it_kx_k - \overline \Lambda_k\mathbf
it_k\Lambda_k) \).  The terms
\( -\tfrac12\overline\Lambda_k\mathbf it_k\Lambda_k \) ensure that the
sum in \( \mu_\Lambda \) converges, but otherwise are arbitrary linear
terms in \( t_k \) with values in \( \iH \).  With our definition of
\( \lambda_k \), we have
\begin{equation*}
  \mu_\Lambda(x) = \sum_{k\in\bL} \bigl(\lambda_k + \tfrac12\overline
  x_k\mathbf ix_k \bigr)e_k^*.
\end{equation*}
The hyperKähler moment map for the subgroup \( N_\beta \) is then
\( \mu_\beta = \iota^*\mu_\Lambda \colon M_\Lambda \to \iH \otimes
\lie n_\beta^* \).  We define
\begin{equation*}
  M = M(\beta,\lambda) = M_\Lambda \hkq N_\beta = \mu_\beta^{-1}(0)
  / N_\beta.
\end{equation*}
Since \eqref{eq:ex-lie} is exact, we have dually
\( \ker\iota^* = \imm\beta^* \) and hence the following
characterisation of \( \mu_\beta^{-1}(0) \).

\begin{lemma}
  A point \( x \in M_\Lambda \) lies in the zero set of the
  hyperKähler moment map~\( \mu_\beta \) for~\( N_\beta \) if and only
  if there is an \( a \in \iH \otimes (\bR^n)^* \) with
  \begin{equation}
    \label{eq:im-beta}
    a(u_k) = \lambda_k + \tfrac12\overline x_k\mathbf ix_k
  \end{equation}
  for each \( k \in \bL \), where \( u_k = \beta(e_k) \).
  \qed
\end{lemma}

In equation \eqref{eq:im-beta}, note that \( x_k = 0 \) if and only if
\( a(u_k) = \lambda_k \).  Indeed, the map \( \bH \to \iH \),
\( v \mapsto \overline v\mathbf iv \), is surjective with
\( \abs{\overline v\mathbf iv} = \abs v^2 \).  As in
\cite{Bielawski-Dancer:toric}, we define affine subspaces
\( H_k \subset \iH \otimes (\bR^n)^* \) of real codimension~\( 3 \) by
\begin{equation}
  \label{eq:flat}
  H_k = H(u_k,\lambda_k)
  = \Set{a \in \iH \otimes (\bR^n)^* \with a(u_k) = \lambda_k}
\end{equation}
which we call \emph{flats}.  Note that \( T^n = T_\lambda/N_\beta \)
acts on~\( M \) and that if \( M \) is smooth this action preserves
induced the hyperKähler structure and has moment
map~\( \phi\colon M \to \iH \otimes (\bR^n)^* \) induced
by~\( \mu_\Lambda \): indeed
\( \nLie T^n = \lie t_\lambda/\lie n_\beta \) implies
\( (\nLie T^n)^* = (\lie n_\beta)^0 \), the annihilator of
\( n_\beta \) in~\( \lie t_\lambda^* \), so on~\( \mu_\beta^{-1}(0) \)
the map~\( \mu_\Lambda \) takes values in
\( (\lie n_\beta)^0 = (\nLie T^n)^* \) and descends to~\( M \)
as~\( \phi \).  It follows, as in~\cite{Bielawski-Dancer:toric}, that
\( \phi \) induces a homeomorphism \( M/T^n \to \iH\otimes(\bR^n)^* \)
and that for \( p \in M \), the stabiliser \( \stab_{T^n}(p) \) is the
subtorus with Lie algebra spanned by the \( u_k \) such that
\( \phi(p) \in H_k \).

\begin{theorem}
  \label{thm:construction}
  Suppose \( u_k = \beta(e_k) \in \bZ^n \), \( k \in \bL \), are
  primitive and span~\( \bR^n \).  Let \( \lambda_k \in \iH \),
  \( k \in \bL \), be given such that the convergence
  condition~\eqref{eq:convergence} holds and the flats
  \( H_k = H(u_k,\lambda_k) \), \( k \in \bL \), are distinct.  Then
  the hyperKähler quotient \( M = M(\beta,\lambda) \) is smooth if
  \setitemindent{\hspace{\parindent}(a)}
  \begin{enumerate}[\upshape (a)]
  \item\label{item:a} any set of \( n+1 \) flats~\( H_k \) has empty
    intersection, and
  \item\label{item:b} whenever \( n \) distinct flats
    \( H_{k(1)},\dots,H_{k(n)} \) have non-empty intersection the
    corresponding vectors \( u_{k(1)},\dots,u_{k(n)} \) form a
    \( \bZ \)-basis for~\( \bZ^n \).
  \end{enumerate}
\end{theorem}

We break the proof in to several steps.

\begin{proposition}
  Suppose \( u_k \in \bZ^n \), \( k\in\bL \), are primitive,
  span~\( \bR^n \) and satisfy condition~\itref{item:b} of
  Theorem~\ref{thm:construction}.  Then
  \( \mathcal U = \Set{u_k\with k\in\bL} \) is finite.
\end{proposition}

\begin{proof}
  Note that if \( u_{k(1)},\dots,u_{k(n)} \) are linearly independent
  then \( \bigcap_{j=1}^n H_{k(j)} \) is a single point.  Thus
  \itref{item:b} implies that \( \mathcal U \) contains a
  \( \bZ \)-basis \( v_1,\dots,v_n \) for \( \bZ^n \).  Then matrix
  \( A \) with columns \( v_1,\dots,v_n \) is invertible with inverse
  in~\( M_n(\bZ) \), so \( A \) lies in
  \( \GL(n,\bZ) = \Set{B\in M_n(\bZ) \with \det B = \pm1} \).
  Multiplying with \( A^{-1} \) we may thus assume for the purpose of
  this proof that \( \mathcal U \) contains the standard basis
  \( \mathbf e_1,\dots,\mathbf e_n \).

  Suppose \( u = (u_1,\dots,u_n) \in \mathcal U \) is different from
  \( \mathbf e_i \), for all \( i=1,\dots,n \).  For
  \( j \in \Set{1,\dots,n} \), consider the matrix \( A_j \) with
  columns
  \( u,\mathbf e_1,\dots,\widehat{\mathbf e_j},\dots,\mathbf e_n \),
  so \( \mathbf e_j \) is omitted.  We have \( \det A_j = \pm u_j \).
  If \( \det A_j \) is non-zero, then its columns are linearly
  independent and the discussion above gives \( \det A_j = \pm 1 \).
  It follows that \( u_j \in \Set{-1,0,1} \) for
  each~\( j \in \Set{1,\dots,n}\).  In particular, there are only
  finitely many such \( u \)'s.
  \qed
\end{proof}

It follows that under condition~\itref{item:b}, the set
\( \Set{\norm{u_k} \with k\in\bL} \) is bounded, so
\eqref{eq:conv-lambda} and \eqref{eq:convergence} are equivalent.  Now
the map \( \beta\colon \lie t_\lambda \to \bR^n \) has Riesz
representation
\( \beta(t) = \inp{t}\gamma_{\lambda,\lie t} = \sum_{k\in\bL} (1 +
\abs{\lambda_k}) t_k\gamma_k \) given by
\( \gamma_k = u_k/(1+\abs{\lambda_k}) \).  As
\( \norm{\gamma}_{\lambda,\lie t}^2 = \sum_{k\in\bL}
(1+\abs{\lambda_k})^{-1}\norm{u_k}^2 \), boundedness
of~\( \norm{u_k} \) and \eqref{eq:conv-lambda} show that \( \gamma \)
lies in~\( \lie t_\lambda \).  Thus \( \beta \) is continuous and it
follows that \( N_\beta \) is a Hilbert subgroup of~\( T_\lambda \) of
codimension~\( n \).

\begin{lemma}
  If conditions \itref{item:a} and \itref{item:b} hold then the group
  \( N_\beta \) acts freely on \( \mu_\beta^{-1}(0) \).
\end{lemma}

\begin{proof}
  Given a subset \( \bK \subset \bL \), we define the subgroup \(
  T_{\bK}\) subgroup of \( T_\lambda \) by
  \begin{equation*}
    T_{\bK} = \Set{(g_k)_{k\in\bL} \in T_\lambda \with g_\ell = 1 \ \forall
    k \notin \bK},
  \end{equation*}
  so that the Lie algebra of \( T_{\bK} \) is spanned by
  \( \Set{e_k\with k\in \bK} \).  For \( x \in M_\Lambda \), the
  stabiliser \( \stab_{T_\lambda}(x) \) is \( T_{\bK} \) where
  \( \bK = \Set{k\with x_k=0} \).

  Now consider \( x \in \mu_\beta^{-1}(0) \).
  Equation~\eqref{eq:im-beta} implies that \( x_k = 0 \) if and only
  if \( \phi(\pi(x)) \in H_k \), where
  \( \pi\colon \mu_\beta^{-1}(0) \to M \) is the quotient map.  Thus
  \( \stab_{T_\lambda}(x) = T_{\bK(x)} \), where
  \( \bK(x) = \Set{k\with \phi(\pi(x)) \in H_k} \).
  Condition~\itref{item:a} implies that \( \bK(x) \) contains at
  most~\( n \) elements. The stabiliser of \( x \) under~\( N_\beta \)
  consists of those elements in \( \stab_{T_\lambda}(x) \) that lie in
  the kernel of the map \( T_\lambda \to T^n \) induced
  by~\( \beta \).  But \( \beta(e_k) = u_k \) and implies that
  \( (g_k)_{k\in\bL} \in T_{\bK(x)} \) maps to
  \( h = \prod_{k\in\bK(x)} g_k\exp(u_k) = \prod_{k\in\bK(x)}
  \exp(i\theta_ku_k) \in T^n \), where \( g_k = e^{i\theta_k} \).

  Condition~\itref{item:b} implies that the \( u_k \),
  \( k\in \bK(x) \), are part of a \( \bZ \)-basis for~\( \bZ^n \), so
  we may change basis via an element of \( \GL(n,\bZ) \) so that the
  \( u_k \) become the first~\( r \) basis elements.  Then
  \( h \)~becomes
  \( \diag(e^{i\theta(1)},\dots,e^{i\theta(r)},1,\dots,1) \), where
  \( \theta(j) \) is a corresponding relabelling of the
  \( \theta_k \)'s.  It follows that \( \theta_k \in 2\pi\bZ \) and so
  \( g_k = 1 \) for each \( k\in \bK(x) \).  Thus
  \( \stab_{N_\beta}(x) \) is trivial, as claimed.
  \qed
\end{proof}

As in \cite{Goto:A-infinity}, for \( x \in M_\Lambda \), let
\( X_x\colon \lie t_\lambda \to T_xM_\Lambda \) be the map sending an
element to the corresponding tangent vector at~\( x \) generated by
the action:
\begin{equation*}
  X_x(t) = \left.\frac d{ds} (\exp(st)x)\right|_{s=0} = \bigl(\mathbf i
  t_kx_k\bigr)_{k\in\bL}.
\end{equation*}

\begin{lemma}
  For \( x \in \mu_\beta^{-1}(0) \), the map \( X_x \) induces a
  linear homeomorphism from \( \lie n_\beta \) to the tangent space \(
  T_x(N_\beta\cdot x) \) of the \( N_\beta \)-orbit through~\( x \).
\end{lemma}

\begin{proof}
  For general \( x \in M_\Lambda \), we have
  \( \norm{X_x(t)}^2 = \norm{(t_kx_k)_{k\in\bL}}^2 = \sum_{k\in\bL}
  \abs{t_k}^2\abs{{\Lambda_k + v_k}}^2 \).  Except for finitely many
  \( k\in\bL \), we have \( \abs{\Lambda_k} \geqslant 2\abs{v_k} \),
  so for these \( k \), we have
  \( \abs{\Lambda_k + v_k}^2 \leqslant \abs{3\Lambda_k/2}^2 = 9
  \abs{\lambda_k} / 8 \).  It follows that there is a
  constant~\( C_x \), independent of~\( t \), such that
  \( \norm{X_x(t)} \leqslant C_x \norm{t}_{\lambda_,\lie t} \).  Thus
  \( X_x\colon \lie t_\lambda \to T_xM_\Lambda \) is continuous.

  Now let
  \( \bK_1 = \bL \setminus \Set{k \in \bL \with \abs{\Lambda_k}
  \geqslant 1 \geqslant 2\abs{v_k}} \), which is a finite set
  by~\eqref{eq:conv-lambda} and the condition that
  \( v\in\bL^2(\bH) \).  For \( k \notin \bK_1 \), we have
  \( \abs{\lambda_k} \geqslant 1/2\), so
  \( \abs{x_k}^2 = \abs{\Lambda_k + v_k}^2 \geqslant
  \abs{\Lambda_k/2}^2 = \abs{\lambda_k}/8 \geqslant
  (1+\abs{\lambda_k})/32 \).  It follows that for
  \( k \notin \bK(x) = \Set{k\with x_k=0} \), there is a constant
  \( c_x > 0 \) such that
  \( \abs{t_kx_k}^2 \geqslant c_x(1+\abs{\lambda_k})\abs{t_k}^2 \).

  For \( x \in \mu_\beta^{-1}(0) \), the set \( \bK(x) \) coincides
  with the previous definition
  \( \bK(x) = \Set{k\with \phi(\pi(x)) \in H_k} \) and so contains at
  most \( n \) elements.  Let
  \( V_x = \Span\Set{e_k\with k\in \bK(x)} \leqslant \lie t_\lambda \)
  and write \( \pr^\bot\colon \lie t_\lambda \to V_x^\bot \) for the
  orthogonal projection away from~\( V_x \).  Then \( \beta \) is
  injective on~\( V_x \), so \( \pr^\bot \) is a continuous linear
  bijection
  \( \pr_\beta\colon \lie n_\beta \to \pr^\bot(\lie n_\beta) \).  The
  image is the orthogonal complement to
  \( V_x \oplus \beta^\dag(\beta(V_x)^\bot) \), where
  \( \beta(V_x)^\bot \) is the orthogonal complement in~\( \bR^n \).
  As \( \beta \) is surjective, its adjoint~\( \beta^\dag \) is
  injective, so \( \pr^\bot(\lie n_\beta) \) is of finite codimension
  and thus a Hilbert subspace of~\( \lie t_\lambda \).  By the Open
  Mapping Theorem, we conclude that \( \pr_\beta^{-1} \) is continuous,
  and we note that its norm is non-zero.

  Now for \( x \in \mu_\beta^{-1}(0) \) and \( t \in \lie n_\beta \),
  we have
  \begin{equation*}
    \begin{split}
      \norm{X_x(t)}^2
      = \norm{(t_kx_k)_{k\in\bL}}^2
      &\geqslant
        c_x\sum_{k\notin\bK(x)}(1+\abs{\lambda_k})\abs{t_k}^2
        = c_x \norm{\pr_\beta(t)}_{\lambda,\lie t}^2\\
      &\geqslant \frac{c_x}{\norm{\pr_\beta^{-1}}^2}
        \norm{t}_{\lambda,\lie t}^2,
    \end{split}
  \end{equation*}
  showing that \( X_x \) has continuous inverse on
  \( T_x(N_\beta\cdot x) \).
  \qed
\end{proof}

It now follows, as in \cite{Goto:A-infinity}, that for
\( x \in \mu_\beta^{-1}(0) \), the differential
\( d\mu_\beta\colon T_xM_\Lambda \to \iH \otimes \lie n_\beta^* \) is
split, with right inverse the \( \bR \)-linear map given by
\( \mathbf a\otimes \delta = AX_x(t) \), where
\( \delta = \inp t\any_{\lambda,\lie t} \) and
\( A = a_1I + a_2J + a_3K \) for
\( \mathbf a = a_1\mathbf i + a_2\mathbf j + a_3\mathbf k \).  This
implies that \( \mu_\beta^{-1}(0) \) is a smooth Hilbert submanifold
of~\( M_\Lambda \).  On \( \mu_\beta^{-1}(0) \), Goto's construction
\cite{Goto:A-infinity} of slices~\( S_x \) goes through unchanged: one
considers the map \( F_x\colon \mu_\beta^{-1}(0) \to \lie n_\beta^* \)
given by \( F_x(x+w)(t) = \inp w{X_x(t)} \) and puts
\( S_x = F_x^{-1}(0) \cap U \), for a sufficiently small
neighbourhood~\( U \) of~\( x \).  Thus
\( M(\beta,\lambda) = \mu_\beta^{-1}(0) / N_\beta \) is a smooth
manifold.

Fix a point
\( q \in \iH\otimes(\bR^n)^* \setminus\bigcup_{k\in\bL} H_k \).  For
\( x \in \mu_\beta^{-1}(0) \) with \( \phi\pi(x) = q \), we have that
\( x_k \ne 0 \) for all \( k \in \bL \).  Thus \( T_\lambda \) acts
freely on~\( x_k \).  As \( e_k \in \lie t_\lambda \) for
each~\( k\in\bL \), it follows that
\( e_k \in F = T_x(T_\lambda\cdot x) \), and that
\( T_XM_\Lambda = F \oplus IF \oplus JF \oplus KF \).  As \( \beta \)
is surjective, we conclude that \( F/T_x(N_\beta\cdot x) \) is of
dimension~\( n \) and that
\( M(\beta,\lambda) = \mu_\beta^{-1}(0)/N_\beta \) is of
dimension~\( 4n \).

The standard considerations of the hyperKähler quotient construction
shows that \( M(\beta,\lambda) \) inherits a smooth hyperKähler
structure, completing the proof of Theorem~\ref{thm:construction}.

Just as in Hattori \cite{Hattori:Ainfty-volume}, one may use the
\( T_\lambda \) action to show that different choices of
\( (\Lambda_k)_{k\in\bL} \) yielding the same
\( (\lambda_k)_{k\in\bL} \) result in hyperKähler structures that are
isometric via a tri-holomorphic map.

\subsection{Classification of complete hypertoric manifolds}
\label{sec:class-hypert-manif}

Now suppose that \( M \) is an arbitrary complete connected hypertoric
manifold of dimension~\( 4n \) and write \( G = T^n \).  Bielawski
\cite[\S4]{Bielawski:tri-Hamiltonian} shows that locally \( M \) has
much of the structure of the hypertoric manifolds constructed above.

Indeed for each \( p \in M \), we may find a \( G \)-invariant
neighbourhood of the form \( U = G \times_H W \), where
\( H = \stab_G(p) \) and \( W = T_p(G\cdot p)^\bot \subset T_xM \).
Now \( H \)~acts trivially on \( W_1 = (\iH) T_p(G\cdot p) \), and
effectively as an Abelian subgroup of \( \SP(r) \) on the orthogonal
complement \( W_2 = W \cap W_1^\bot \cong \bH^r \).  Counting
dimensions, it follows that \( H \)~acts as \( T^r \) on~\( W_2 \),
and hence \( H \) is connected.  The image of the singular orbits
in~\( U \) is a union of distinct flats \( H_k = H(u_k,\lambda_k) \)
as in~\eqref{eq:flat}, where \( u_k \) may be chosen to lie in
\( \bZ^n \subset \bR^n = \lie g \) and be primitive vectors.  The
collection \( \Set{u_k\with \mu(p) \in H_k} \) is then part of a
\( \bZ \)-basis for~\( \bR^n \) spanning the Lie algebra of~\( H \).
Furthermore, examining the structure of~\( \mu \) on such a
neighbourhood~\( U \), Bielawski shows that \( \mu \)~induces a local
homeomorphism \( M/G \to \iH \otimes \lie g^* \cong \bR^{3n} \).  In
particular, the hyperKähler moment map
\( \mu\colon M \to \iH \otimes \lie g^* \) is an open map.

Let \( \Set{H_k\with k\in \bL} \) be the collection of all flats that
arise in this way.  The index set~\( \bL \) is finite or countably
infinite, since \( M \) is second countable.

\begin{lemma}
  \label{lem:alpha-a}
  Suppose \( \alpha \in (\bZ^n)^* \subset (\bR^n)^* = \lie g^* \) is
  non-zero.  Let \( T_\alpha \) be the subtorus of \( G = T^n \) whose
  Lie algebra is spanned by
  \( \ker\alpha = \Set{u \in \lie g \with \alpha(u) = 0} \).  For
  \( a \in \iH \otimes \lie g^* \), write
  \( [a]_\alpha = a + \iH\otimes \bR\alpha \) for the equivalence
  class of \( a \) in \( \iH\otimes (\lie g^*/\bR\alpha) \).  Then
  except for countably many choices of~\( [a]_\alpha \), the
  group~\( T_\alpha \) acts freely on~\( \mu^{-1}([a]_\alpha) \).
\end{lemma}

Note that \( (\ker\alpha)^* = \lie g^*/\bR\alpha \).

\begin{proof}
  Consider the intersection \( [a]_\alpha \cap H_k \).  A general
  point of \( [a]_\alpha \) is \( a + q\otimes\alpha \),
  \( q \in \iH \), which lies in \( H_k = H(u_k,\lambda_k) \) only if
  \( a(u_k) + q\alpha(u_k) = \lambda_k \).  If \( \alpha(u_k) \ne 0 \)
  this equation has a unique solution for \( q \); if
  \( \alpha(u_k) = 0 \) then there is a solution only if
  \( a(u_k) = \lambda_k \) and then \( [a]_\alpha \subset H_k \).
  Thus choosing \( a(u_k) \ne \lambda_k \) for each \( k\in \bL \),
  ensures that \( [a]_\alpha \cap H_k \) is empty for every \( k \)
  for which \( u_k \in \ker\alpha \).  It follows that~\( T_\alpha \)
  acts almost freely on~\( \mu^{-1}([a]_\alpha) \), but as each
  stabiliser of the \( T^n \)-action is connected, we find that
  \( T_\alpha \) acts freely.
  \qed
\end{proof}

\begin{corollary}
  For \( \alpha \), \( T_\alpha \) and \( a \) as in
  Lemma~\ref{lem:alpha-a}, the hyperKähler quotient
  \( M(a,\alpha) = \mu^{-1}([a]_\alpha) / T_\alpha \) is a complete
  hyperKähler manifold of dimension four.  Furthermore,
  \( M(a,\alpha) \) carries an effective tri-Hamiltonian circle
  action.
\end{corollary}

\begin{proof}
  As \( T_\alpha \) is compact and acts freely on
  \( \mu^{-1}([a]_\alpha) \), it follows that \( M(a,\alpha) \) is
  hyperKähler \cite{Hitchin-KLR:hK}.  Completeness of~\( M \) implies
  completeness of the level set \( \mu^{-1}([a]_\alpha) \) and hence
  of the hyperKähler quotient.

  As \( \mu \) is~\( T^n \)-invariant, the level set
  \( \mu^{-1}([a]_\alpha) \) is preserved by~\( T^n \) and we get an
  action of the circle \( T^n/T_\alpha \) on the quotient.
  Identifying \( [a]_\alpha \) with \( \iH \) via
  \( a + q\otimes \alpha \mapsto q \), the restriction of \( \mu \) to
  \( \mu^{-1}([a]_\alpha) \) descends to a hyperKähler moment map for
  this action.
  \qed
\end{proof}

From the four-dimensional classification Theorem~\ref{thm:4d}, we find
have that the moment map of~\( M(a,\alpha) \) surjects on
to~\( \iH \).  Interpreting this in terms of the moment map \( \mu \)
of~\( M \), we have that \( [a]_\alpha \) lies in the image
of~\( \mu \).  But on~\( M \) the moment map \( \mu \) is an open
map. And, as
\( \bigcup \Set{[a]_\alpha\with a(u_k)\ne \lambda_k \forall k: u_k \in
\ker\alpha} \) is dense in~\( \iH\otimes\bR^n \), we conclude that
\( \mu \) is a surjection.

Furthermore, the metric on \( M(a,\alpha) \) is given by a potential
of the form~\eqref{eq:V3} up to an overall positive scale.  The
elements of \( Z \) are just the intersection points of
\( [a]_\alpha \cap H_k \), for \( u_k \notin \ker\alpha \).  These are
the points \( a_k = a + q_k\otimes \alpha \) with
\( q_k = (\lambda_k - a(u_k))/\alpha(u_k) \).  Now for
\( p \in \mu^{-1}([a]_\alpha) \), writing
\( \mu(p) = a + q \otimes \alpha \), we have
\begin{equation*}
  \begin{split}
    \mu(p) - a_k
    &= (q - q_k)\otimes \alpha
    = \frac{q\alpha(u_k) - (\lambda_k - a(u_k))}{\alpha(u_k)} \otimes
    \alpha\\
    &= (\inp{\mu(p)}{u_k} - \lambda_k) \otimes \frac\alpha{\alpha(u_k)}.
  \end{split}
\end{equation*}
Thus the potential for \( M(a,\alpha) \) is proportional to
\begin{equation*}
  V_\alpha(p) = c + \frac12 \sum_{k\in\bL} \frac
  1{\norm{\inp{\mu(p)}{u_k}-\lambda_k}}
  \frac{\abs{\alpha(u_k)}}{\norm\alpha},
\end{equation*}
where we may include the terms with \( u_k \in \ker\alpha \), since
they contribute zero, and \( \norm\alpha \) is the norm of \( \alpha
\) with respect to the standard inner product from the
identification~\( \lie g = \bR^n \).

Using this inner product we may identify \( \lie g\) with
\( \lie g^* \).  Then the function
\begin{equation*}
  r_k(b) = \norm{b(u_k)-\lambda_k}/\norm{u_k} = \norm{b(\hat u_k) -
  \hat \lambda_k},
\end{equation*}
where \( \hat u_k = u_k/\norm{u_k} \) and
\( \hat\lambda_k = \lambda_k/\norm{u_k} \), corresponds to the
distance of \( b \) from the flat~\( H_k \).  We may thus write
\begin{equation}
  \label{eq:V-alpha}
  V_\alpha(p) = c + \frac12 \sum_{k\in\bL}
  \frac{\abs{\hat\alpha(\hat u_k)}}{r_k(\mu(p))},
\end{equation}
for \( \hat\alpha = \alpha/\norm\alpha \).

Now choose \( a \) so that \( a(u_k) \ne \lambda_k \) for all
\( k \in \bL \).  Then for \( p \in \mu^{-1}(a) \) we have that
\( V_\alpha \) in~\eqref{eq:V-alpha} is finite for each non-zero
integral \( \alpha \in \lie g^* \).  As each unit vector
\( \hat u \in \lie g \cong \bR^n \) has
\( \inp{\hat u}{\mathbf e_i} \geqslant 1/\sqrt n \) for some
\( i \in \Set{1,\dots,n} \), we conclude that
\( \sum_{k\in\bL} 1/r_k(\mu(p)) \) converges.  In particular the
distance of \( \mu(p) \) to \( H_k \) is bounded below by a uniform
constant.  It follows that there is an open neighbourhood~\( U \) of
\( \mu(p) \) in \( \mu(M) \subset \iH\otimes\lie g^* \) for which
\( a(u_k) \ne \lambda_k \) for all \( a \in U \).

Let \( M_U \) be a connected component of~\( \mu^{-1}(U) \).  Then
\( T^n \) acts freely on~\( M_U \) and the hyperKähler structure on
\( M_U \) is uniquely determined via a polyharmonic function \( F \)
on \( U \subset \iH\otimes\lie g^* \) as follows.  The hyperKähler metric
is of the form
\begin{equation*}
  g = \sum_{i,j=1}^n (V^{-1})_{ij}\beta_0^i\beta_0^j +
  V_{ij}(\alpha_I^i\alpha_I^j + \alpha_J^i\alpha_J^k + \alpha_K^i\alpha_K^j),
\end{equation*}
where for \( X_1,\dots,X_n \) is a basis for \( \bR^n = \lie g \), we
have \( \alpha_A^i = X_i \hook \omega_A \),
\( (V^{-1})_{ij} = (g(X_i,X_j)) \) and \( \beta_0^1,\dots,\beta_0^n \)
are the \( C^\infty(M_U) \)-linear combinations of
\( g(X_1,\any),\allowbreak \dots,\allowbreak g(X_n,\any) \) such that
\( \beta_0^i(X_j) = \delta^i_j \).  As a result of Pedersen and Poon
\cite{Pedersen-Poon:Bogomolny} and the Legendre transform of Lindström
and Ro\v cek \cite{Lindstrom-R:scalar,Hitchin-KLR:hK} the functions
\( V_{ij} \) on \( U \subset \iH \otimes \lie g^* \) are
\emph{polyharmonic} meaning that they are harmonic on each affine
subspace \( a + \iH \otimes \bR\alpha \),
\( \alpha \in \lie g^*\setminus\Set0 \).  Furthermore this matrix of
functions is given by a single polyharmonic function
\( F\colon U \to \bR \) via \( V_{ij} = F_{x_ix_j} \), where we choose
a unit vector \( \mathbf e \in \iH \) and \( (x_1,\dots,x_n) \) are
standard coordinates on
\( \bR^n = \bR\mathbf e\otimes\lie g^* \subset \iH\otimes\lie g^* \).
We write
\begin{equation*}
  s_k(b) = \inp{\mathbf e}{b(\hat u_k) - \hat\lambda_k}.
\end{equation*}
As \( \mathbf e \) acts on \( \mathbf e^\bot \subset \iH \) as a
complex structure, we may choose corresponding standard complex
coordinates \( (z_1,\dots,z_n) \) on
\( \mathbf e^\bot \otimes \lie g^* \).  A potential \( V \) of the
form~\eqref{eq:V3} is then \( V = F_{xx} \) with
\begin{equation*}
  F(x,z) = \frac14 c(2x^2 - \abs z^2) + \frac12\sum_{k\in\bL}
  (s_k\log(s_k+r_k) - r_k).
\end{equation*}
As in Bielawski \cite{Bielawski:tri-Hamiltonian}, we now deduce that
for \( M_U \) the function~\( F \) has the form
\begin{equation*}
  F = \sum_{k\in \bL} a_k(s_k\log(s_k+r_k) - r_k) + \sum_{i,j=1}^n
  c_{ij}(4x_ix_j - z_i\overline z_j - z_j\overline z_i)
\end{equation*}
for some real constants \( a_k \), \( c_{ij} \).  As Bielawski
explains the \( c_{ij} \) terms are from a \emph{Taub-NUT deformation}
of a metric determined by the first sum, that is there is a hypertoric
manifold~\( M_2 \) and
\( M_U = M_2 \times (S^1\times\bR^3)^m \hkq T^m \) with \( T^m \)
acting effectively on the product of \( S^1 \)-factors, trivially on
the \( \bR^3 \)-factors and as a subgroup of~\( T^n \) on~\( M_2 \).
By analyticity, the hyperKähler metric on \( M_U \) determines the
hyperKähler metric on~\( M \).

Using Bielawski's techniques and the computations of
\cite{Bielawski-Dancer:toric}, one may now conclude that~\( M_2 \)
comes from the construction of the previous section.  In particular,
we note that the \( a_k \)'s are bounded and convergence of
\( V_\alpha(p) \) in~\eqref{eq:V-alpha} for each non-zero
integral~\( \alpha \) corresponds to the
condition~\eqref{eq:convergence}.  To see this first we remark that
convergence of the \( V_\alpha(p) \)'s corresponds to convergence of
\( R(b) = \sum_{k\in\bL}1/r_k(b) \) for \( b = \mu(p) \): this follows
from \( \abs{\hat\alpha(\hat u_k)} \leqslant 1 \) and
\( \inp{\hat u_k}{\mathbf e_i} \geqslant 1/\sqrt n \) for
some~\( i \).  Now
\( r_k(b) \leqslant \abs{b(u_k)} + \abs{\hat\lambda_k} \leqslant
(1+\norm b)(1+\abs{\hat\lambda_k}) \) gives that convergence
of~\( R(b) \) implies~\eqref{eq:convergence}.  Conversely, note
that~\eqref{eq:convergence} implies that
\( \abs{\hat\lambda_k} < 1 +2\norm b \) for only finitely many
\( k\in\bL \).  But for
\( \abs{\hat\lambda_k} \geqslant 1+2\norm b \), we have
\( r_k(b) \geqslant (\abs{\hat\lambda_k}+1)/2 \), as
in~\eqref{eq:ineq}, so we get convergence of~\( R(b) \).

We have thus proved the following result.

\begin{theorem}
  Let \( M \) be a connected hypertoric manifold of
  dimension~\( 4n \).  Then \( M \) is a product
  \( M = M_2 \times (S^1\times\bR^3)^m \) with \( M_2 \) a hypertoric
  manifold of the type constructed in~\S\ref{sec:constr-hypert-manif},
  i.e., the hyperKähler quotient of flat Hilbert hyperKähler manifold
  by an Abelian Hilbert Lie group.  The hyperKähler metric in \( M \)
  is either the product hyperKähler metric or a Taub-NUT deformation
  of this metric.
  \qed
\end{theorem}

From the proof and the construction of the previous section, we have
the following properties of the hyperKähler moment map of in this
situation.

\begin{corollary}
  If \( M \) is a connected complete hyperKähler manifold of
  dimension~\( 4n \) with an effective tri-Hamiltonian action of~\(
  T^n \), then the hyperKähler moment map \( \mu\colon M \to \iH
  \otimes (\bR^n)^* \cong \bR^{3n} \) is surjective with connected
  fibres.
  \qed
\end{corollary}

\section{Cuts and modifications}

Symplectic cutting was introduced by Lerman in 1995
\cite{Lerman:cuts}. The construction starts with a symplectic manifold
with Hamiltonian circle action, and produces a new manifold of the
same type, but with different topology.  Explicitly, given $M$ with
circle action and associated moment map $\mu$, we form the symplectic
quotient at level~$\epsilon$,
\[
  \Mcut^{\epsilon} = (M \times \C) \symp_{\epsilon} S^1
\]
where the $S^1$ is the antidiagonal of the product action obtained
from the given action on $M$ and the standard rotation on $\C$. Note
that the moment map for the action on $\C$ is
$\phi \colon z \mapsto |z|^2$.

The new space $\Mcut$ is of the same dimension as $M$ and
inherits a circle action from the diagonal action on $M \times
\C$. Moreover, as
\[
  \Mcut = \Set{(m,z) \with \mu(m) - |z|^2 = \epsilon} / S^1
\]
we see that the points $\Set{ m \with \mu (m) < \epsilon}$ are
removed. Moreover, because $\phi \colon z \mapsto |z|^2$ is a trivial
circle fibration over $(0, \infty)$ with the circle fibre collapsing
to a point at the origin, we see that the region
$\Set{ m \with \mu(m) > \epsilon }$ remains unchanged, while the
hypersurface $\mu^{-1} (\epsilon)$ is collapsed by a circle action.

We can generalise this to the case of torus actions, by replacing $\C$
by a toric variety associated to a polytope $\Delta$. The region
$\mu^{-1} (\R^n \setminus (\Delta + \epsilon))$ will be removed, the
preimage of the $\epsilon$-translate of the interior of $\Delta$ is
unchanged, while collapsing by tori takes place on the preimage of
lower-dimensional faces of the translated polytope.

For general geometries, we want to mimic this construction by looking at the
appropriate quotient of $M \times N$ by an Abelian group~$G$, where
$N$ is a space whose reduction by $G$ is a point. The topological
change in $M$ will be controlled by the geometry of the moment map for
the $G$ action on $N$.

The simplest example is that of a hyperK\"ahler manifold with circle
action.  We now explain the hyperK\"ahler analogue of a cut in this
situation, which we call a \emph{modification} \cite{Dancer-S:mod}.
The natural choice of $N$ is now the quaternions $\HH$. Several new
features now emerge, because the hyperKähler moment map
$\mu \colon \HH \rightarrow \R^3$ has very different properties from
that for $\C$. In particular, it is \emph{surjective}, and is a
\emph{non-trivial} circle fibration over $\R^3 \setminus \Set0$; in
fact $\phi$ is given by the Hopf map on each sphere.  This means that
in forming the hyperK\"ahler quotient
$\Mmod = (M \times \HH) \hkq S^1$ we do \emph{not} discard any points
in $M$ (hence the use of the terminology modification rather than
cut!). We still collapse the locus $\mu^{-1}(\epsilon)$ by a circle
action, because $\phi$ is injective over the origin. The complements
$M \setminus \mu^{-1} (\epsilon)$ and
$\Mmod \setminus (\mu^{-1}(\epsilon)/S^1)$ can no longer be
identified, because of the non-triviality of the Hopf
fibration. Instead, the topology has been given a `twist'. An example
of this is if we start with $M = \HH$.  Now iterating the above
construction generates the Gibbons-Hawking $A_k$ multi-instanton
spaces, where the spheres at large distance are replaced by lens
spaces $S^3 / \Z_{k+1}$.

We can make this more precise by observing that the space
\[
  M_1 = \Set{(m, q) \in M \times \HH \with \mu(m) - \phi(q) = \epsilon}
\]
projects onto both $M$ and $\Mmod$. The first map is just projection
onto the first factor (onto as $\phi$ is surjective), while
the second is just the quotient map $M_1 \rightarrow \Mmod$.
Note that the first map is not quite a fibration: it has $S^1$ fibres
generically but over $\mu^{-1}(\epsilon)$ the fibres collapse to a
point.

On the open sets where both maps are fibrations, note that
\( M_1 \to \Mmod \) is a Riemannian submersion, but that the
projection to first factor is not.  As shown in
\cite{Swann:twist-mod}, the metric~\( \tilde g \) induced on~\( M \)
by \( M_1 \) has the form
\begin{equation}
  \label{eq:g-mod}
  \tilde g = g + V(\mu)g_{\bH},
\end{equation}
where \( g_{\bH} = \alpha_0^2 + \alpha_I^2 + \alpha_J^2 + \alpha_K^2
\), with \( \alpha_0 = g(X,\any) \), etc., and \( V(\mu) =
1/2\norm{\mu - \epsilon} \) is the potential of the flat hyperKähler
metric on~\( \bH \).

One may now generalise the hyperKähler modification, by replacing
\( \bH \) by any hypertoric manifold~\( N \) of dimension~\( 4 \).
The modification changes the metric as in~\eqref{eq:g-mod} with
\( V(\mu) \) now the potential function of~\( N \), so one of the
functions~\eqref{eq:V3}.  In \cite{Swann:twist-mod} it is proved that
metric changes of the form~\eqref{eq:g-mod} with \( V \) now an
arbitrary smooth invariant function on~\( M \), so called `elementary
deformations', only lead to new hyperKähler metrics when \( V \) is
\( \pm V(\mu) \) for some hypertoric~\( N^4 \).  The case of negative
\( V \) corresponds precisely to inverting a modification via a
positive~\( V \).

For general torus actions one can take $N$ to be a hypertoric manifold
and we get a similar picture to that above.

\section{Non-Abelian moment maps}

One can also consider cutting constructions for non-Abelian group
actions.  Now, because the diagonal and anti-diagonal actions no
longer commute, one considers the product of a $K$-manifold $M$ with a
space $N$ with $K \times K$ action and then reduces by the
antidiagonal action formed from the action on $M$ and (say) the left
action on $N$.

In the symplectic case, following Weitsman \cite{Weitsman:cuts}, with
$K = \Un(n)$ one can take $N = \Hom(\C^n, \C^n)$ with $K \times K$
action $A \mapsto UA V^{-1}$. The moment map for the right action is
$\mu \colon A \mapsto i A^* A$, with image $\Delta$ the set of
non-negative Hermitian matrices. We have a picture that is quite
reminiscent of the Abelian case, essentially because the fibres of
$\phi$ are orbits of the \emph{left} action. The right moment map
$\phi$ gives a trivial fibration with fibres $\Un(n)$ over the interior
of the image, while over the lower-dimensional faces of $\Delta$
(corresponding to non-negative Hermitian matrices that are not strictly
positive), the fibres are $\Un(n)/\Un(n-k)$ where $k$ is the number of
positive eigenvalues. This gives a nice non-Abelian generalisation of
the toric cuts described above. To form the cut space we remove the
complement of $\mu^{-1} (\Delta + \epsilon)$ and perform collapsing by
the appropriate unitary groups on the preimages of the
lower-dimensional faces of $\Delta + \epsilon$.

In the hyperK\"ahler case life becomes more complicated.  An obvious
choice of $N$ is $\Hom(\C^n, \C^n) \oplus \Hom(\C^n, \C^n)^*$, with
action
\[
  (U, V) \colon (A, B) \mapsto (UAV^{-1}, VBU^{-1}).
\]
The hyperKähler moment map for the right $\Un(n)$-action is now
\[
  \mu\colon (A, B) \mapsto (\frac{\mathbf i}{2} (A^* A - B B^* ), BA)
  \in \iH \otimes \un(n)^* \cong \un(n) \oplus \gl(n, \C).
\]
In contrast to the symplectic case, or the Abelian hyperK\"ahler case,
the fibres of $\mu$ are no longer group orbits in general; in
particular the left $U(n)_L$ action need not be transitive and indeed
the quotient of a fibre by this action may have positive
dimension. The result is that when we perform the non-Abelian
hyperK\"ahler modification, blowing up of certain loci occurs.

This is closely related to the phenomenon that the fibres of
hyperK\"ahler moment maps over \emph{non-central} elements may have
larger than expected dimension, even on the locus where the group
action is free. This is because the kernel of the differential of a
moment map $\mu$ is the orthogonal of
$I {\mathcal G} + J {\mathcal G} + K {\mathcal G}$, where
$\mathcal G$, as in \S 2.1, denotes the tangent space to the
orbits of the group action.  For the fibre over
a central element this sum is direct (because $\mathcal G$ is tangent
to the fibre so is orthogonal to the sum, and hence the three summands
are mutually orthogonal), but over non-central elements this is no
longer necessarily true. The dimension of the fibre is now no longer
determined by the dimension of $\mathcal G$, and hence not determined
by the dimension of the stabiliser. Note that in the K\"ahler
situation the kernel of $d \mu$ is just the orthogonal of
$ I \mathcal G$, so here the dimension of the fibre is completely
controlled by the dimension of the stabiliser, even in the non-Abelian
case.

Another example of the unexpected behaviour of hyperK\"ahler moment
maps over non-central elements is the phenomenon of disconnected
fibres.  This is in contrast to symplectic moment maps, whose level
sets enjoy many connectivity properties, see \cite{Sjamaar:convexity}
and the references therein.  As a simple hyperK\"ahler example we may
consider the $SU(2)$ action on
$\HH^2 = \C^2 \times (\C^2)^* = T^*\bC^2$
\begin{equation}
  \label{eq:action-TCn}
  A \colon (z,w) \mapsto (Az, wA^{-1})
\end{equation}
with hyperKähler moment map
\begin{equation}
  \label{eq:mu-sun}
  \mu = (\mu_{\R}, \mu_{\C}) \colon (z,w) \mapsto \Bigl(\frac{\mathbf
  i}{2} (zz^\dag -w^\dag w)_0,  (zw)_0\Bigr),
\end{equation}
where $\any_0$ denotes trace-free part, and \( \any^\dag \) the
conjugate transpose. (This calculation arose in discussions with
S. Tolman).

We are interested in finding the fibre of $\mu$ over
$(\alpha, \beta) \in \su(2) \oplus \sln(2, \C) \cong \iH \otimes
\su(2)^*$. Using the $SU(2)$-equivariance, we may take
\[
  \beta =
  \begin{pmatrix}
    \lambda & \mu \\
    0       & -\lambda \\
  \end{pmatrix}
  .
\]
We find that the fibre is:
\begin{enumerate}[\upshape (1)]
\item empty or a disjoint pair of circles if $\lambda, \mu$ are both
  non-zero;
\item empty, a circle or a disjoint pair of circles if $\lambda$ is
  zero and $\mu$ non-zero;
\item empty or a disjoint pair of circles if $\mu$ is zero and
  $\lambda$ non-zero;
\item a disjoint pair of circles or a point if $\lambda = \mu = 0$
  (the point fibre occurs exactly over the origin
  $\alpha = \beta = 0$.
\end{enumerate}

Now let us turn to the the case of \( \SU(3) \).  This acts on
\( T^*\bC^3 = \bC^3\oplus (\bC^3)^* \) via \eqref{eq:action-TCn} and
has the same formula~\eqref{eq:mu-sun} for the hyperKähler moment
map~\( \mu\colon T^*\bC^3 \to \iH \otimes \su(3)^*\cong \su(3) \oplus
\Sl(3,\bC) \).  In particular \( \mu_\bC(z,w) \) has off-diagonal
\( (i,j) \)-entries \( z_iw_j \) whilst the diagonal entries are of
the form \( 2z_iw_i - \sum_{k\ne i} z_kw_k \).  Similarly, for
\( j>i \),
\( (\mu_\bR(z,w))_{ij} = \tfrac{\mathbf i}2
(z_i\overline{z_j}-\overline{w_i}w_j) \) and
\begin{equation*}
  (\mu_\bR(z,w))_{ii} = \frac{\mathbf i}6
  \Bigl\{2(\abs{z_i}^2-\abs{w_i}^2) - \sum_{k\ne
  i}(\abs{z_k}^2-\abs{w_k}^2)\Bigr\}.
\end{equation*}

We consider the fibre \( \mu^{-1}(\alpha,\beta) \), for
\( \alpha\in\su(3) \), \( \beta\in\Sl(3,\bC) \).  Note that
\( \dim_\bC\Sl(3,\bC) = 8 \) is strictly greater than
\( \dim_\bC T^*\bC^3 = 6 \) so there are now restrictions on
\( \beta \) to lie in the image of~\( \mu_\bC \).  In particular, we
see that \( \mu \) is not surjective.

For more detail, note that the map \( \mu \) is
\( \SU(3) \)-equivariant and we may use the action to put \( \beta \)
in to the canonical upper triangular form
\begin{equation*}
  \beta =
  \begin{pmatrix}
    \lambda_1&\xi_1&\zeta\\
    0&\lambda_2&\xi_2\\
    0&0&\lambda_3
  \end{pmatrix}
  ,
\end{equation*}
with \( \lambda_1+\lambda_2+\lambda_3=0 \).  This gives the three
constraints
\begin{equation}
  \label{eq:zero}
  z_2w_1 = 0,\quad z_3w_1=0,\quad z_3w_2 = 0.
\end{equation}

\paragraph{Case 1}
\label{sec:case-1}
Both \( \xi_i\ne0 \): we have \( z_1w_2\ne0\ne z_2w_3 \) implying that
\( z_1,z_2, w_2,w_3 \) are non-zero and so \( \zeta \ne 0 \).
Equation~\eqref{eq:zero} gives \( w_1 = 0 = z_3\).  This implies
\begin{equation*}
  \lambda_1 = \lambda_3 = -\tfrac12\lambda_2 = -\tfrac13 z_2w_2.
\end{equation*}
Now \( z_1 \) determines the remaining variables via
\begin{equation*}
  w_2 = \frac{\xi_1}{z_1},\quad w_3 = \frac \zeta{z_1},\quad z_2 =
  \frac{\xi_2}{w_3} = \frac{\xi_2}\zeta z_1
\end{equation*}
giving the relation
\begin{equation*}
  3\lambda_2\xi_2 = 2\xi_1\zeta,
\end{equation*}
so \( \lambda_2\ne0 \).  Constraints on \( z_1 \) come from
\( \mu_\bR \) and using the above relations they are seen to only
involve linear combinations of \( x=\abs{z_1}^2 \) and \( 1/x \).
Thus there are at most \( 2 \) values for \( \abs{z_1} \).  However,
closer inspection reveals that the entries above the diagonal
in~\( \mu_\bR \) are
\begin{equation*}
  \begin{pmatrix}
    *&p\abs{z_1}^2&0\\
    *&*&q/\abs{z_1}^2\\
    *&*&*
  \end{pmatrix}
  ,
\end{equation*}
with \( p = \overline{\xi_2/\zeta} \), \( q = \overline \zeta \xi_1 \).
Thus \( \alpha_{13}=0 \), \( 4\alpha_{12}\alpha_{23} = - pq = -
\overline{\xi_2}\xi_1 \) and there is at most one solution for \(
\abs{z_1} \), which together with \( \beta \) specifies the diagonal
entries of \( \alpha \).  Thus the fibre is either a circle or empty.

\paragraph{Case 2}
\label{sec:case-2}
\( \xi_1\ne0,\xi_2=0 \): we have \( z_1\ne0\ne w_2 \) and thus \(
z_3=0 \).  Now \( z_2w_1=0=z_2w_3 \) divides into cases.

(a) \( z_2\ne0,w_1=0=w_3 \): gives
\( \lambda_2 = \tfrac23 z_2w_2 = -2\lambda_1 = -2\lambda_3 \ne 0 \).
\( z_2 = 3\lambda_2/2w_2 \), \( z_1 = \xi_1/w_2 \).  In \( \alpha \),
the \( (1,2) \)-entry determines \( z_1\overline z_2 \ne 0 \) which
specifies \( \abs{w_2} \) uniquely.  So the fibre is either a circle
or empty.

(b) \( w_1\ne0, z_2=0 \): gives \( \lambda_1 = \tfrac23 z_1w_1
= -2\lambda_2 = -2\lambda_3 \ne 0 \), \( \zeta = z_1w_3 \).  So \( w_1 =
3\lambda / 2z_1 \), \( w_2 = \xi_1 / z_1 \), \( w_3 = \zeta / z_1 \).
The \( (1,2) \) entry of \( \alpha \) then specifies \( \abs{z_1} \)
uniquely and the fibre is a single circle.

(c) \( w_1=0=z_2 \): gives \( \lambda_i=0 \) for all \( i \) and
\( \zeta = z_1w_3 \), so \( w_2 = \xi_1/z_1 \), \( w_3 = \zeta/z_1 \).
For \( \zeta\ne0 \), \( \abs{z_1} \) is determined by
\( \alpha_{23} \) and the fibre is a circle or empty.  For
\( \zeta = 0 \), the only non-zero entries are the diagonal ones, with
a the first entry \( -2 \) times the other two which are equal.  These
give a single quadratic equation for \( \abs{z_1} \), so the fibre is
either \( 2 \) circles, \( 1 \) circle or empty.  For example, the
first diagonal entry in \( \mu_\bR \) is proportional to
\( 2\abs{z_1}^2 + \abs{\xi_1}^2/\abs{z_1}^2 \), which attains any
sufficiently large positive value at two different values of
\( \abs{z_1} \).  In particular, the fibre of \( \mu \) can be
disconnected.

\paragraph{Case 3}
\label{sec:case-3}
\( \xi_1=0=\xi_2=\zeta \): there are three types of case:

(a) two \( z \)'s non-zero: \( z_1\ne0\ne z_2 \) implies \( w \equiv 0
\) and \( \beta=0 \).  The off-diagonal entries of \( \alpha \)
determine \( z_2 \) and \( z_3 \) in terms of \( z_1 \).  If \(
\alpha_{23}\ne0 \), then this determines \( \abs{z_1} \) and the fibre
is empty or a circle.  Otherwise the diagonal entries lead to a
quadratic constraint in \( \abs{z_1} \).  The fibres are thus \( 2 \)
circles, one circle or empty.

(b) one \( z \) and one \( w \) non-zero: then these must have the
same index, say \( z_1\ne0\ne w_1 \).  So \( w_2=0=w_3=z_2=z_3 \).  \(
\beta \) is diagonal with two repeated eigenvalues, \( w_1 =
-3\lambda_1/2z_1 \).  \( \alpha \) is necessarily diagonal, the
diagonal entries give a quadratic constraint on \( \abs{z_1} \).  The
fibres are \( 2 \) circles, one circle or empty.

(c) one \( z \) or \( w \) non-zero: gives \( \beta=0 \), \( \alpha \)
diagonal with the sign of the entries in \( i\alpha  \) determined by
whether it is \( z \) or \( w \) that is non-zero.  The fibre is
either a circle, a point or empty.

\section{Implosion}

Implosion arose as an abelianisation construction in symplectic
geometry \cite{Guillemin-JS:implosion}.  Given a Hamiltonian
$K$-manifold $M$, one forms a new Hamiltonian space $M_{\mathrm{impl}}$
with an action of the maximal torus $T$ of $K$, such that the symplectic
reductions agree
\[
  M \symp_{\lambda} K = M_{\mathrm{impl}} \symp_{\lambda} T
\]
for $\lambda$ in the closed positive Weyl chamber. In most cases the
implosion is a singular stratified space, even in $M$ is smooth.

The key example is the implosion $(T^*K)_{\mathrm{impl}}$ of $T^*K$ by
(say) the right $K$ action.  Because $T^*K$ has a $K \times K$ action,
the implosion has a $K \times T$ action, and in fact we may implode a
general Hamiltonian $K$-manifold by forming the reduction of
$M \times (T^*K)_{\mathrm{impl}}$ by the diagonal $K$ action.  In this
sense $(T^*K)_{\mathrm{impl}}$ is a universal example for imploding
$K$-manifolds. Concretely, $(T^*K)_{\mathrm{impl}}$ is obtained from the
product of $K$ with the closed positive Weyl chamber $\bar{\tf}^*$ by
stratifying by the face of the Weyl chamber (i.e.\ by the centraliser
$C$ of points in $\bar{\tf}^*$) and then collapsing by the commutator
of~$C$.

There is also a more algebraic description of the universal implosion
$(T^*K)_{\mathrm{impl}}$, as the non-reductive Geometric Invariant Theory
(GIT) quotient $K_{\C} \symp N$, where $N$ is the maximal unipotent
subgroup.  Note that in general non-reductive quotients need not exist
as varieties, due to the possible failure of finite generation for the
ring of $N$-invariants. In the above case (and in the hyperK\"ahler
and holomorphic symplectic situation described below) it is a
non-trivial result that we do have finite generation, so the quotient
does exist as an affine variety.

In a series of papers
\cite{Dancer-KS:implosion,Dancer-KS:hypertoric,Dancer-KS:twistor} the
authors and Kirwan described an analogue of implosion in hyperK\"ahler
geometry. There is a hyperK\"ahler metric (due to Kronheimer
\cite{Kronheimer:cotangent}) with $K \times K$ action on the cotangent
bundle $T^* K_{\C}$, and the idea is that the analogue of the
universal symplectic implosion should be the complex-symplectic
quotient (in the GIT sense) of $T^*K_{\C}$ by $N$.  Explicitly, this
quotient is $(K_{\C} \times \n^{\circ}) \symp N$, which it is often
convenient to identify with $(K_{\C} \times \bmf ) \symp N$, where
$\bmf$ is the Borel subalgebra.

In the case of $K= SU(n)$ it was shown in \cite{Dancer-KS:implosion}
that $(K_{\C} \times \n^{\circ}) \symp N$ arises, via a quiver
construction, as a hyperK\"ahler quotient with residual $K \times T$
action. Moreover the hyperK\"ahler quotients by $T$ may be identified
with the Kostant varieties, that is the subvarieties of $\kf_{\C}$
obtained by fixing the values of a generating set of invariant
polynomials. For example, reducing at zero gives the nilpotent
variety. For general semi-simple $K$, results of Ginzburg and Riche
\cite{Ginzburg-R:affine-Grassmannian} give the existence of the
complex-symplectic quotient $(K_{\C} \times \n^{\circ}) \symp N$ as an
affine variety, and the complex-symplectic quotients by the $T_{\C}$
action again give the Kostant varieties.

There is a link here with some intriguing work by Moore and Tachikawa
\cite{Moore-T:TQFTs}. They propose a category $HS$ whose objects are
complex semi-simple groups, and where elements of
$\Mor (G_1, G_2)$ are complex-symplectic manifolds with
$G_1 \times G_2$ action (together with a commuting circle action that
acts on the complex-symplectic form with weight $-2$). We compose
morphisms $X \in \Mor(G_1, G_2)$ and $Y \in \Mor (G_2, G_3)$
by taking the complex-symplectic reduction of $X \times Y$ by the
diagonal $G_2$ action.  The Kronheimer space $T^*G$, where $G=K_{\C}$,
gives a canonical element of $\Mor (G, G)$---in fact this
functions as the identity in $\Mor (G, G)$.
The implosion now may be viewed as giving an element of $\Mor(G, T_{\C})$,
where $G= K_{\C}$ and $T_{\C}$ is the complex maximal torus in $G$.

\begin{acknowledgement}
  Andrew Swann partially supported by the Danish Council for
  Independent Research, Natural Sciences. We thank Sue Tolman for
  discussions about the disconnected fibres of hyperK\"ahler moment
  maps.
\end{acknowledgement}


\end{document}